\documentclass[12pt,twoside,a4paper]{article}
\usepackage[utf8]{inputenc}
\usepackage{amsfonts}
\usepackage{amssymb}
\usepackage{amsthm}
\usepackage[pdftex]{graphicx}
\usepackage[bookmarksopen=false,pdftex=true,breaklinks=true,%
      backref=page,pagebackref=true,plainpages=false,%
      hyperindex=true,pdfstartview=FitH,%
      pdfpagelabels=true,colorlinks=true,linkcolor=blue,%
      citecolor=red,urlcolor=red,hypertexnames=false,colorlinks=false%
      ]%
   {hyperref}
 
\usepackage{mathrsfs}
\usepackage{mathtools}
\usepackage{stmaryrd}
\usepackage[all]{xy}
\usepackage{makeidx}
\usepackage[french]{babel}
\usepackage{xspace}

\usepackage{etoolbox}
\makeatletter\def\th@plain{\slshape}
\patchcmd{\th@remark}{\itshape}{\slshape}{}{}\makeatother

\newcounter{bidon}
\newcommand{\rdb}{\refstepcounter{bidon}}

\newtheorem{theorem}{Théorème}[section]

\newtheorem{lemma}[theorem]{Lemme}

\newtheorem{plcc}[theorem]{Principe local-global concret}

\theoremstyle{definition}
\newtheorem{definition}[theorem]{Définition}

\newtheorem{definota}[theorem]{Définition et notation}

\theoremstyle{remark}

\newcounter{MF}
\newcommand\stMF{\stepcounter{MF}}

\newcommand{\lec}{\stMF\ifodd\value{MF}lecteur\xspace\else 
lectrice\xspace\fi}

\newcommand{\lecs}{\stMF\ifodd\value{MF}lecteurs\xspace\else 
lectrices\xspace\fi}

\newcommand{\alec}{\stMF\ifodd\value{MF}au lecteur\xspace\else%
à la lectrice\xspace\fi}

\newcommand{\dlec}{\stMF\ifodd\value{MF}du lecteur\xspace\else%
de la lectrice\xspace\fi}

\newcommand{\llec}{\stMF\ifodd\value{MF}le lecteur\xspace\else la lectrice\xspace\fi}

\newcommand{\Llec}{\stMF\ifodd\value{MF}Le lecteur\xspace\else La lectrice\xspace\fi}

\newcommand{\lui}{\ifodd\value{MF}lui\xspace\else
elle\xspace\fi}

\newcommand{\celui}{\ifodd\value{MF}celui\xspace\else
celle\xspace\fi}

\newcommand{\ceux}{\ifodd\value{MF}ceux\xspace\else
celles\xspace\fi}

\newcommand{\er}{\ifodd\value{MF}er\xspace\else
ère\xspace\fi}

\newcommand{\eux}{\ifodd\value{MF}eux\xspace\else
elles\xspace\fi}

\newcommand{\eUx}{\ifodd\value{MF}eux\xspace\else
euse\xspace\fi}

\newcommand{\eUX}{\ifodd\value{MF}eux\xspace\else
euses\xspace\fi}

\newcommand{\leux}{\ifodd\value{MF}leux\xspace\else
leuse\xspace\fi}

\newcommand{\il}{\ifodd\value{MF}il\xspace\else
elle\xspace\fi}

\newcommand{\ien}{\ifodd\value{MF}ien\xspace\else
ienne\xspace\fi}

\newcommand{\e}{\ifodd\value{MF}\xspace \else e\xspace\fi}

\newcommand{\n}{\ifodd\value{MF}n\xspace\else nne\xspace\fi}

\makeatletter
\newcommand{\la}{\@ifstar{\ifodd\value{MF}le\else
la\fi}{\stMF\ifodd\value{MF}le\else la\fi}}
\makeatother

\newcommand\note{\rdb
\noi{\sl Note. }}

\newcommand \thref[1] {théorème~\ref{#1}}

\newcommand\oge{\leavevmode\raise.3ex\hbox{$\scriptscriptstyle\langle\!\langle\,$}}
\newcommand\feg{\leavevmode\raise.3ex\hbox{$\scriptscriptstyle\,\rangle\!\rangle$}}

\newcommand\gui[1]{\oge{#1}\feg}

\renewcommand\paragraph[1]{

\medskip \noindent $\bullet$ \textbf{#1}}

\newcommand \recu {récurrence\xspace}

\newcommand \cad {c'est-à-dire\xspace}

\newcommand \ssi {si, et seulement si, }

\newcommand \spdg {sans perte de généralité\xspace}

\newcommand \Propeq {Les propriétés suivantes sont 
équivalentes.}

\newcommand \Kev {$\gK$-\evc}

\newcommand \Lev {$\gL$-\evc}

\newcommand \klg {$\gk$-\alg}
\newcommand \klgs {$\gk$-\algs}

\newcommand \Klg {$\gK$-\alg}
\newcommand \Klgs {$\gK$-\algs}

\newcommand \Amo {$\gA$-module\xspace}

\newcommand \Bmo {$\gB$-module\xspace}

\newcommand \kmo {$\gk$-module\xspace}

\newcommand \Ali {application $\gA$-\lin}
\newcommand \Alis {applications $\gA$-\lins}

\newcommand \kli {application $\gk$-\lin}

\newcommand \Bli {application $\gB$-\lin}
\newcommand \Blis {applications $\gB$-\lins}

\newcommand \afr {anneau \frl}

\newcommand \afrvr {\afr avec \ravs}

\newcommand \agN {\alg de Newton\xspace}

\newcommand \agq{algébrique\xspace}
\newcommand \agqs{algébriques\xspace}

\newcommand \alg {algèbre\xspace}
\newcommand \algs {algèbres\xspace}

\newcommand \ali {application \lin}

\newcommand \apf {\alg \pf}

\newcommand \arc {anneau réel clos\xspace}

\newcommand \cara{caractéristique\xspace}

\newcommand \carn{caractérisation\xspace}

\newcommand \cdi{corps discret\xspace}
\newcommand \cdis{corps discrets\xspace}

\newcommand \com {comaximaux\xspace}

\newcommand \demo{démonstration\xspace}     
\newcommand \demos{démonstrations\xspace}

\newcommand \dfn{définition\xspace}  
\newcommand \dfns{définitions\xspace}

\newcommand \dile{différentielle\xspace}  
\newcommand \diles{différentielles\xspace}

\newcommand \discri{discriminant\xspace}

\newcommand \dvn {dérivation\xspace}
\newcommand \dvns {dérivations\xspace}

\newcommand \eal {\und {étale}\xspace}
\newcommand \eals {\und {étales}\xspace}

\newcommand \eds {extension des scalaires\xspace}

\newcommand \eco {\elts \com}

\newcommand \egmt {également\xspace}

\newcommand \egt {égalité\xspace}

\newcommand \elr{élémentaire\xspace}

\newcommand \elt{élément\xspace}  
\newcommand \elts{éléments\xspace}

\newcommand\evc{espace vectoriel\xspace}

\newcommand \eqve {équivalente\xspace}  
  
\newcommand \eqves {équivalentes\xspace}  

\newcommand \eqvc {équivalence\xspace}

\newcommand \esid {essentiellement identique\xspace}  
\newcommand \esids {essentiellement identiques\xspace}

\newcommand \eseq {essentiellement \eqve}  
\newcommand \eseqs {essentiellement \eqves}

\newcommand \fit {fidèlement\xspace}

\newcommand \fpte {\fit plate\xspace}

\newcommand \frl {fortement réticulé\xspace}

\newcommand \frls {fortement réticulés\xspace}

\newcommand\gnle{générale\xspace}

\newcommand \grl{groupe \rtl}
\newcommand \grls{groupes \rtls}

\newcommand \homo {homomorphisme\xspace}

\newcommand \idm {idempotent\xspace}
\newcommand \idms {idempotents\xspace}

\newcommand \inteq {intuitivement \eqve}
\newcommand \inteqs {intuitivement \eqves}

\newcommand \ird {irréductible\xspace}

\newcommand \iso {isomorphisme\xspace}

\newcommand \itf {idéal \tf}

\newcommand \iv {inversible\xspace}

\newcommand \lgbe {locale-globale\xspace}

\newcommand \lin {linéaire\xspace}
\newcommand \lins {linéaires\xspace}

\newcommand \mo {mono\"{\i}de\xspace}

\newcommand \pf {de \pn finie\xspace}

\newcommand \pn {présentation\xspace}

\newcommand \pol {polynôme\xspace}
\newcommand \pols {polynômes\xspace}

\newcommand \polu {\pol \unt}

\newcommand \poll{polynomial\xspace}

\newcommand \polmin {\pol minimal\xspace}

\newcommand \prmt {précisément\xspace}

\newcommand \pro {projectif\xspace}
\newcommand \pros {projectifs\xspace}

\newcommand \ptf {\pro \tf}
\newcommand \ptfs {\pros \tf}

\newcommand \ravs {racines virtuelles\xspace}

\newcommand \rtl {réticulé\xspace}
\newcommand \rtls {réticulés\xspace}

\newcommand \sfio {système fondamental d'\idms orthogonaux\xspace}

\newcommand \spb {séparable\xspace}  

\newcommand \spl {séparable\xspace}  
\newcommand \spls {séparables\xspace}

\newcommand \stf {strictement fini\xspace}

\newcommand \stfe {strictement finie\xspace}
\newcommand \stfes {strictement finies\xspace}

\newcommand \syp {\sys \poll}

\newcommand \sys {système\xspace}

\newcommand \tf {de type fini\xspace}

\newcommand \tho {théorème\xspace}

\newcommand \unt {unitaire\xspace}

\newcommand \uvle {univer\-selle\xspace}

\newcommand \zed {z\'{e}ro-dimensionnel\xspace}
\newcommand \zede {z\'{e}ro-dimensionnelle\xspace}
\newcommand \zeds {z\'{e}ro-dimensionnels\xspace}

\newcommand \zedr {\zed réduit\xspace}
\newcommand \zedre {\zede réduite\xspace}
\newcommand \zedrs {\zeds réduits\xspace}

\newcommand \cof {constructif\xspace}

\newcommand \covs {constructives\xspace}

\newcommand {\junk}[1]{}

\newcommand{\Cadre}[2]{%

\medskip%
\newskip\oldleftskip
\newskip\oldrightskip
\oldleftskip=\leftskip%
\oldrightskip=\rightskip%
\leftskip=-\tabcolsep%
\rightskip=-\tabcolsep%
\begin{center}\fbox{%
\begin{tabular}%
{p{#1\textwidth}}
\setlength{\parindent}{5mm}%
\vspace{-1.5mm}#2\vspace{1mm}%
\end{tabular}}\end{center}\par\medskip%
\leftskip=\oldleftskip%
\rightskip=\oldrightskip%
\setlength{\parindent}{6mm}}

\newcommand{\Encadre}[1]{\Cadre{.8}{#1}}

\newcommand \noi {\noindent}

\renewcommand \leq{\leqslant}

\renewcommand \geq{\geqslant}

\newcommand \etoz{$^*$}
\newcommand \ist{_\star}

\newcommand \Ati {\gA^{\!\times}}

\newcommand \aqo[2] {#1\sur{\gen{#2}}\!}
\newcommand \Aqo[2] {#1\sur{\big\langle{#2}\big\rangle}\!}

\newcommand \frt[1] {\!\left|_{#1}\right.\!}

\newcommand \tra[1] {{\,^{\rm t}\!#1}}
\newcommand \gen[1] {\left\langle{#1}\right\rangle}

\newcommand \so[1] {\left\{\,{#1}\, \right\}}

\newcommand \sur[1] {\!\left/#1\right.}
\newcommand \und[1] {\underline{#1}}

\newcommand\Ae[1]{\gA^{\!#1}}

\newcommand \dex[1] {[\,#1\,]}

\newcommand{\mt}{\mapsto}

\newcommand{\llongrightarrow}{\relbar\joinrel\mkern-1mu\longrightarrow}

\newcommand\vers[1]{\buildrel{#1}\over \longrightarrow }
\newcommand\vvers[1]{\buildrel{#1}\over \llongrightarrow }

\newcommand \lora {\longrightarrow}

\renewcommand \leq{\leqslant}
\renewcommand \geq{\geqslant}

\newcommand \som {\sum\nolimits}

\newcommand \te {\otimes}

\newcommand\tsbf[1]{\textsf{\textbf{\textup{#1}}}}

\newcommand\Tsbf[1]{\hyperref[Ax#1]{\tsbf{#1}}}

\newcommand\Sa[1]{\hyperref[theorie#1]{\sa{#1}}}
\newcommand\sa[1]{\hbox{\usefont{T1}{pzc}{m}{it}#1}\,}

\newcommand\Om[2]  {\Omega_{{#2}/{#1}}}
\newcommand\Der[3] {{\rm Der}_{{#1}}({#2},{#3})}

\newcommand \eoe {\hbox{}\nobreak\hfill
\vrule width .5em height .5em depth 0mm \par \smallskip}

\newcommand \ov[1] {\overline{#1}}

\newcommand \wi[1] {\widetilde{#1} }

\makeatletter
\newif\if@borderstar
\def\bordercmatrix{\@ifnextchar*{%
  \@borderstartrue\@bordercmatrix@i}{\@borderstarfalse\@bordercmatrix@i*}%
}
\def\@bordercmatrix@i*{\@ifnextchar[{%
  \@bordercmatrix@ii}{\@bordercmatrix@ii[()]}
}
\def\@bordercmatrix@ii[#1]#2{%
  \begingroup
    \m@th\@tempdima.875em\setbox\z@\vbox{%
      \def\cr{\crcr\noalign{\kern 2\p@\global\let\cr\endline}}%
      \ialign {$##$\hfil\kern.2em\kern\@tempdima&\thinspace%
      \hfil$##$\hfil&&\quad\hfil$##$\hfil\crcr\omit\strut%
      \hfil\crcr\noalign{\kern-\baselineskip}#2\crcr\omit%
      \strut\cr}}%
    \setbox\tw@\vbox{\unvcopy\z@\global\setbox\@ne\lastbox}%
    \setbox\tw@\hbox{\unhbox\@ne\unskip\global\setbox\@ne\lastbox}%
    \setbox\tw@\hbox{%
      $\kern\wd\@ne\kern-\@tempdima\left\@firstoftwo#1%
        \if@borderstar\kern.2em\else\kern -\wd\@ne\fi%
      \global\setbox\@ne\vbox{\box\@ne\if@borderstar\else\kern.2em\fi}%
      \vcenter{\if@borderstar\else\kern-\ht\@ne\fi%
        \unvbox\z@\kern-\if@borderstar2\fi\baselineskip}%
\if@borderstar\kern-2\@tempdima\kern.4em\else\,\fi\right\@secondoftwo#1 $%
    }\null\;\vbox{\kern\ht\@ne\box\tw@}%
  \endgroup
}
\makeatother

\newcommand{\Dpp}[2]{{{\partial #1}\over{\partial #2}}}

\makeatletter
\def\revddots{\mathinner{\mkern1mu\raise\p@
\vbox{\kern7\p@\hbox{.}}\mkern2mu
\raise4\p@\hbox{.}\mkern2mu\raise7\p@\hbox{.}\mkern1mu}}
\makeatother

\newcommand\Pnv[9]{
$$\quad\quad\quad\quad
\vcenter{\xymatrix@C=1.5cm
{
#1 \ar[d]_{#2} \ar[dr]^{#3} \\
{#4} \ar@{-->}[r]_{{#5}\,!}   & {#6} \\
}}
\quad\quad
\vcenter{
\hbox{\small {#7}}
\hbox{~\\[1mm] ~}
\hbox{\small {#8}}
\hbox{~\\[-1mm] ~}
\hbox{\small {#9}}
\hbox{~\\[0mm] ~ }}
$$
}

\newcommand\PNV[9]{
$$\quad\quad\quad\quad
\vcenter{\xymatrix@C=1.5cm
{
#1 \ar[d]_{#2} \ar[dr]^{#3} \\
{#4} \ar@{-->}[r]_{{#5}\,!}   & {#6} \\
}}
\quad\quad
\vcenter{
\vspace{4mm}
\hbox{\small {#7}}
\hbox{~\\[-1.7mm] ~}
\hbox{\small {#8}}
\hbox{~\\[-1.7mm] ~}
\hbox{\small {#9}}
\hbox{~\\[2mm] ~ }
}
$$
}

\newcommand \NN{\mathbb {N}}

\newcommand \gk {\mathbf{k}}

\newcommand \gA {\mathbf{A}}
\newcommand \gB {\mathbf{B}}
\newcommand \gC {\mathbf{C}}

\newcommand \gK {\mathbf{K}}

\newcommand \gL {\mathbf{L}}

\newdimen\xyrowsp
\newcommand{\SCO}[6]{
\xymatrix @R = \xyrowsp {
                                  &1 \ar@{-}[dl] \ar@{-}[dr] \\
#3 \ar@{-}[ddr]                   &   & #6 \ar@{-}[ddl] \\
                                  &\bullet\ar@{-}[d] \\
                                  &\bullet   \\
#2 \ar@{-}[ddr] \ar@{-}[uur]      &   & #5 \ar@{-}[ddl] \ar@{-}[uul] \\
                                  &\bullet \ar@{-}[d] \\
                                  &\bullet  \\
#1 \ar@{-}[uur]                   &   & #4 \ar@{-}[uul] \\
                                  & 0 \ar@{-}[ul] \ar@{-}[ur] \\
}
}

\newcommand \car {\MA{\mathrm{car}}}

\newcommand \Coker {\MA{\mathrm{Coker}}}

\renewcommand \det {\MA{\mathrm{det}}}
\renewcommand \deg {\MA{\mathrm{deg}}}

\newcommand \Disc {\MA{\mathrm{Disc}}}

\newcommand \rg{\MA{\mathrm{rg}}}

\newcommand \JJ {\MA{\mathrm{JAC}}}
\newcommand \rja {\mathrm{Ja}}

\newcommand\MA[1]{\mathop{#1}\nolimits}

\newcommand \cD {{\cal D}}

\newcommand \rd {\mathrm{d}}

\newcommand\fa{\mathfrak{a}}

\newcommand \uf  {{\underline{f}}}

\newcommand \ux {{\underline{x}}}

\newcommand \uX {\underline{X}}

\newcommand \sn {s_1,\ldots,s_n}
\newcommand \xn {x_1,\ldots,x_n}
\newcommand \vr {v_1,\ldots,v_r}

\newcommand \Xn {X_1,\ldots,X_n}

\newcommand \Yr {Y_1,\ldots,Y_r}

\newcommand \kX {{\gk[X]}}
\newcommand \kx {{\gk[x]}}
\newcommand \KX {\gK[X]}
\newcommand \Kx {\gK[x]}
\newcommand \kY {{\gk[Y]}}
\newcommand \KY {\gK[Y]}

\newcommand \Lx {\gL[x]}

\newcommand \KT {\gK[T]}
\newcommand \KuX {\gK[\uX]}

\newcommand \Kux {\gK[\ux]}

\newcommand \KXn {\gK[\Xn]}

\newcommand \kxn {\gk[\xn]}
\newcommand \kXn {\gk[\Xn]}

\newcommand \Kxn {\gK[\xn]}

\newcommand \lfs {f_1,\ldots,f_s}

\DeclareMathAlphabet{\mathpzc}{OT1}{pzc}{m}{it}

\patchcmd{\sectionmark}{\MakeUppercase}{}{}{}
\begin{document}

\stMF

\title{Un \tho décisif}
\author{H. Lombardi}
\maketitle

\begin{abstract} \label{abstract}
Nous donnons une \demo \elr du \tho selon lequel toute \alg nette sur un \cdi est étale, de dimension finie comme \evc et traciquement  étale.
\end{abstract}

\medskip \noindent {\bf Mots clés.}  Algèbre traciquement étale. Algèbre étale.  Algèbre nette. Algèbre nette sur un corps discret. Mathématiques constructives.


\tableofcontents


\section*{Introduction} 
Nous donnons dans cet article une \demo \elr du \tho selon lequel une \alg nette sur un \cdi est étale, de dimension finie et traciquement  étale.

Les \dfns et résultats des sections \ref{secEtaleSurCD} et \ref{sec-dilesKhaler} sont tirés de \cite[Chapitre VI]{ACMC}, ou de la version anglaise \cite[Chapter VI]{CACM}. La section \ref{subsecLNE} donne quelques \dfns et résultats classiques concernant les \algs \pf nettes, lisses ou étales. La section \ref{secthdecisif} établit le \tho fondamental visé.

Comme dans \cite{ACMC} et \cite{CACM}, les \demos ici sont \covs et susceptibles d'être implémentées en Calcul formel.

La démonstration du résultat dans l'article présent est assez proche de celle donnée dans \cite{ACMC} après le \tho VI-6.18, laquelle simplifie la démonstration du corollaire VI-6.15 dans \cite{ACMC} ou \cite{CACM}.  

\smallskip Dans l'article présent $\gk$ désigne un anneau commutatif arbitraire et $\gK$, $\gL$ des \cdis, \cad des anneaux commutatifs non nuls dans lequel tout \elt est nul ou \iv.  

Sur un \cdi, les procédures usuelles de l'\alg \lin sont explicites.

Une extension \agq $\gL$ de $\gK$ est dite \textsl{finie} si c'est un \Kev \tf, \textsl{\stfe} si c'est un \Kev qui possède une base finie.

\section{Algèbres \eals sur un \cdi} \label{secEtaleSurCD}

\begin{definition}\label{defi1Etale}%
Soit~$\gk$ un anneau commutatif et~$\gA$ une \klg. 
On suppose que $\gA$ est un \kmo libre de rang fini $>0$.
\begin{enumerate}
\item L'\alg~$\gA$ est dite  \textsl{traciquement étale}  si
 le \discri $\Disc_{\gA/\gk}$ est \iv. Nous abrégeons
\gui{traciquement étale} en \textsl{\eal} (souligné). 
\end{enumerate}%
Considérons maintenant le cas d'une \alg \stfe sur un \cdi $\gL$%
\begin{enumerate}\setcounter{enumi}{1}
\item Un \elt de~$\gA$ est dit  \textsl{\spl  (sur~$\gL$)} s'il annule un \pol \spl.
\item L'\alg~$\gA$ est dite  \textsl{\agq \spb  (sur~$\gL$)} si tout \elt de~$\gA$ est  \spl sur~$\gL$.
\end{enumerate}%
\end{definition}

\begin{lemma} \label{lem-sanscarre}
Soit $f$ un \polu de~$\gL[X]$ et $\Lx=\aqo{\gL[X]}{f}$ l'\alg 
quotient.
\begin{enumerate}
\item $\Lx$ est \eal \ssi $f$ est \spl.
\item  $\Lx$ est réduite \ssi $f$ est sans facteur carré.
\end{enumerate}
\end{lemma}

\begin{lemma}\label{lemma1Etale} Soit~$\gB\supseteq\gK$ une \alg \stfe.
\begin{enumerate}
\item \label{i1fact1Etale} L'\alg~$\gB$ est \zede\footnote{I.e., pour tout \elt $x$, il existe $n\in\NN$ et $y\in\gB$ tels que $x^n(1-yx)=0$.}. Si $\gB$ est réduite, pour tout $a\in\gB$ il existe un unique \idm $e\in\gK[a]$ tel que~$\gen{a}=\gen{e}$. En outre, lorsque $e=1$, \cad lorsque $a$ est \iv, $a^{-1}\in\gK[a]$.
\item \label{i2fact1Etale} \Propeq
\begin{enumerate}
\item $\gB$ est un \cdi.
\item $\gB$ est sans diviseur de zéro: $xy=0\Rightarrow (x=0$ \textsl{ou} $y=0)$.
\item $\gB$ est connexe\footnote{I.e., tout \idm est égal à $0$ ou $1$.} et réduite.
\item Le \polmin sur~$\gK$ de n'importe quel \elt de~$\gB$ est
\ird.
\end{enumerate}
\item \label{i3fact1Etale} Si~$\gK\subseteq\gL\subseteq\gB$ et si~$\gL$ est 
un \cdi \stf sur~$\gK$, alors~$\gB$ est \stfe sur~$\gL$. En outre,~$\gB$ 
est \eal sur~$\gK$ \ssi elle est \eal
sur~$\gL$ et~$\gL$ est \eal sur~$\gK$.
\item \label{i4fact1Etale} Si $(e_1,\ldots,e_r)$ est un \sfio de~$\gB$, 
$\gB$ est \eal sur
$\gK$ \ssi chacune des compo\-santes~$\gB [1/e_i]$
est \eal sur~$\gK$.
\item \label{i5fact1Etale} Si~$\gB$ est \eal elle est réduite.
\item \label{i6fact1Etale} Si $\car(\gK) >\dex{\gB:\gK}$
et si~$\gB$ est réduite, elle est \eal.
\end{enumerate}
\end{lemma}

\begin{theorem}[\carn des \Klgs \eals]\label{corlemEtaleEtage}~
\\
Soit~$\gB$ une \Klg \stfe donnée sous la forme $\Kxn$.
\Propeq
\begin{enumerate}
\item $\gB$ est \eal sur~$\gK$.
\item  Le \polmin sur~$\gK$ de chacun des
$x_i$ est  \spl.
\item $\gB$ est \agq  \spb sur~$\gK$.
\end{enumerate}
En particulier, un corps~$\gL$ qui est une extension galoisienne de~$\gK$
est \eal sur~$\gK$.
\end{theorem}

\begin{theorem}[\tho de l'\elt primitif] ~
\label{thEtalePrimitif}%
\index{primitif!theoreme@théorème de l'élément ---}
\\
Soit~$\gB$ une \Klg \eal.
\begin{enumerate}
\item  Si~$\gK$ est infini ou si~$\gB$ est un \cdi, alors~$\gB$ est une \alg monogène, \prmt de la forme~$\gK[b]\simeq\aqo{\KT}{f}$ pour un $b\in \gB$ et un $f\in\KT$ \spl. 
%
\item $\gB$ est un produit fini de \Klgs \eals monogènes.
\end{enumerate}
\end{theorem}

\section[Les \diles de Khäler]{Le module des \diles de Khäler}
\label{sec-dilesKhaler}
\Encadre{\noindent Dans la suite de l'article, les \klgs sont toujours supposées \pf.}
\vspace{-1em}

\subsection*{Dérivations}

\begin{definota}\label{defiDeriv}~\\
Soit $\gk$ un anneau commutatif, $\gA$ une \klg et $M$ un \Amo. 
\begin{enumerate}
\item On appelle \textsl{$\gk$-dérivation de $\gA$
dans $M$}, une \kli $\delta$ qui vérifie l'\egt de Leibniz%
\index{derivation@dérivation!d'une algèbre!dans un module}
\[
\delta(ab)=a\delta(b)+b\delta(a).
\]%
\item On note $\Der\gk\gA M$ le \Amo des $\gk$-\dvns de $\gA$ dans $M$.%
\index{derivation@dérivation!module des ---}%
\item Une \dvn de $\gA$  \gui{tout court} est une \dvn à valeurs dans 
$\gA$. Lorsque le contexte est clair,  $\mathrm{Der}(\gA)$ 
est une abréviation pour~$\Der{\gk}{\gA}{\gA}$.%
\index{derivation@dérivation!d'une algèbre}
\end{enumerate}
\end{definota}
\index{module!des diff@des \diles (de Kähler)}%
\index{differentielle (de@\dile (de Kähler)}%
\index{nette!algebre@algèbre ---}%

\note Si $\delta\in \Der\gk\gA M$ et $c\in\gA$, l'\kli $x\mt c\delta(x)$ est aussi une \dvn. Cela justifie le point 2 ci-dessus.\eoe

\begin{lemma} \label{lemdefiDeriv}
Soit $\delta:\gA\to M$ une \dvn, $y\in\gA,(\ux)=(\xn)\in\gA^n$, $f\in\kY$ et $g\in\kXn$. 
\begin{enumerate}
\item On a $\delta(1)=0$ car $1^{2}=1$, et donc $\delta\frt\gk=0$.
\item On a $\delta(y^k)=ky^{k-1}$.
\item On a $\delta(f(y))=f'(y)\delta(y)$.
\item On a $\delta(g(\ux))=\sum_{i=1}^n\Dpp {g}{X_i}(\ux)\delta(x_i) $.
\end{enumerate}
 
\end{lemma}
%
%
%
\subsection*{Le module des différentielles}

On considère maintenant le cas d'une \apf 
\[\gA=\aqo{\kXn}{\lfs}=\kxn=\gk[\ux].\]

On rappelle que la matrice jacobienne du \syp est définie comme\index{jacobien!matrice ---ne}
\[ \JJ_{\uX}(\uf)\; =\;
\bordercmatrix [\lbrack\rbrack]{
    & X_1                     & X_2                     &\cdots  & X_n \cr
f_1 & \Dpp {f_1}{X_1} &\Dpp {f_1}{X_2}  &\cdots  &\Dpp {f_1}{X_n} \cr
f_2 & \Dpp {f_2}{X_1} &\Dpp {f_2}{X_2}  &\cdots  &\Dpp {f_2}{X_n} \cr
f_i & \vdots                  &                         &        & \vdots              \cr
 & \vdots                  &                         &        & \vdots              \cr
f_s & \Dpp {f_s}{X_1} &\Dpp {f_s}{X_2}  &\cdots  &\Dpp {f_s}{X_n} \cr
}
.
\]
Dans la suite, lorsque le contexte est clair, on notera $\rja$ l'\ali définie par la matrice transposée
\[
\rja:=\tra{\JJ_{\uX}(\uf)}:\Ae s\to \Ae n,
\]
et  $(e_1,\dots,e_n)$ la base canonique de~$\Ae n$.
On note $\ov{e_i}$ l'\elt $e_i$ vu dans le quotient $\Coker(\rja)$ de $\Ae n$. On définit 
\[ 
\begin{array}{rcl} 
\rd:\gA\to\Coker(\rja)  &:   & g(\ux)\mapsto   \som_{i=1}^{n} \Dpp {g}{X_i}(\ux)\,\ov{e_i},
 \end{array}
\]

\begin{theorem}[\textsl{dérivation \uvle via la jacobienne}]\label{thDerivUnivPF} 
On utilise les notations précédentes avec la \klg
$
\gA=\aqo{\kXn}{\lfs}=\gk[\ux]. 
$
\begin{enumerate}%
\item L'application $\rd$ est une $\gk$-\dvn avec $\rd(x_i)=\ov{e_i}$.
\item L'application $\rd$ est  une \emph{\dvn \uvle} 
au sens suivant. %
\\
Pour tout \Amo $M$ et
toute $\gk$-\dvn $\delta:\gA\to M$, il existe une unique \Ali $\theta:\Coker(\rja)\to M$
telle que $\theta\circ \rd=\delta$.
\Pnv{\gA}{\rd}{\delta}{\Coker(\rja)}{\theta}{M}{}{$\gk$-\dvns}{\Alis.}
On note \fbox{$\Om{\gk}{\gA}=\Coker(\rja)$}. On l'appelle \emph{le module des \diles de Khäler}.
\end{enumerate}
\end{theorem}%
%

\begin{lemma}[extension des scalaires et \dvn \uvle] \label{lem-Om-eds}~

\noindent  Soit $\gA$ une \klg et un morphisme $\alpha:\gk\to\gk'$ vu comme une extension des scalaires. Notons $\gA'=\alpha\ist(\gA)\simeq \gk'\te_\gk\gA$.\\
Alors on a un \iso naturel\footnote{Plus \prmt, l'application $\gA'$-linéaire $\gA'\to \gk'\te_\gk\Om{\gk}{\gA}$ est une \dvn \uvle pour la $\gk'$-\alg $\gA'$.}  $\Om{\gk'}{\gA'}\simeq\alpha\ist(\Om{\gk}{\gA})\simeq\gk'\te_\gk\Om{\gk}{\gA}$.
\end{lemma}

\begin{lemma}[localisation en haut] \label{propOmkBS} ~

\noindent
 Soient $\gA$ une \klg,  $S$ un \mo de $\gA$ et $M$ un \Amo.
\begin{enumerate}
\item \label{iDerpropOmkBS}%
\begin{enumerate}
\item Une $\gk$-\dvn de $\gA$ dans $M$ induit une unique \dvn de
$\gA_S$ dans $M_S$. Ceci permet de définir un \homo canonique
de $\gA_S$-modules:
$$ 
\varphi:\left(\Der{\gk}{\gA}{M}\right)_S\longrightarrow \Der{\gk}{\gA_S}{M_S}.  
$$
\item L'homomorphisme $\varphi$ est un \iso si $M$ ou $\Om\gk\gA$ est un \Amo \ptf.

\item Si $\gA$ est une une \klg \tf, alors  $\varphi$ est injective. 

\item Si $\gA$ est une \klg \pf, ou encore une \klg \tf et  intègre, alors~$\varphi$ est un \iso.
\end{enumerate}
\item \label{iOmpropOmkBS} On a un \iso naturel de $\gA_S$-modules $(\Om\gk\gA)_S\to\Om\gk{\gA_S}$.
\end{enumerate}

\end{lemma}

\begin{lemma}[extension d'une \alg et \dvns]
\label{propOmABC} ~\\
Soient $\gA$ une \klg, $\gA\vers{\rho}\gB$ une \alg et $M$ un \Bmo\footnote{Malgré le titre du lemme $\rho$ n'est pas supposé injectif.}.
\begin{enumerate}
\item %
\begin{enumerate}
\item Puisque toute $\gA$-\dvn de $\gB$ dans $M$ est aussi
une $\gk$-\dvn on~a une injection $\gB$-\lin canonique
$j:\Der{\gA}{\gB}{M}\to \Der{\gk}{\gB}{M}$.
\item Puisque $M$ est naturellement
un \Amo, on~a une \Bli canonique
\[ 
  \wi \rho:\Der{\gk}{\gB}{M}\lora \Der{\gk}{\gA}{M} \;:\;  \delta\mapsto\delta\circ\rho.
\] 
On peut la noter $\Der\gk\rho M$.
\item La suite d'\Blis ci-dessous est exacte
\[ 
0\;\longrightarrow\;\Der{\gA}{\gB}{M} \;\vvers{j}\;\Der{\gk}{\gB}{M}
\;\vvers{{\wi \rho}}\;\Der{\gk}{\gA}{M}  
\]
En particulier si $\Der{\gk}{\gA}{M}=0$ alors
$\Der{\gA}{\gB}{M}=\Der{\gk}{\gB}{M}$.
\end{enumerate}

\item La suite d'\Blis ci-dessous est exacte
\[
\gB\otimes_\gk \Om{\gk}{\gA}\;\longrightarrow\;\Om{\gk}{\gB}
\;\longrightarrow\;\Om{\gA}{\gB}\;\longrightarrow\;0  
\]
En particulier si $\Om{\gk}{\gA}=0$ alors
$\Om{\gA}{\gB}=\Om{\gk}{\gB}$.
\end{enumerate}
\end{lemma}

\section{Algèbres lisses, nettes, étales} \label{subsecLNE}\label{secLNE}

\begin{definition} \label{defi-LNE} 
On considère une \klg \pf 
\[ 
\gA=\aqo\kXn\lfs
\] et l'on note $\rja$ la transposée de la matrice jacobienne.
\begin{enumerate}
\item La \klg $\gA$ est dite \textsl{nette} ou \textsl{non ramifiée}  lorsque \hbox{$\Om{\gk}{\gA}=0$}, i.e. $\rja$ est surjective, 
ou encore $\cD_{\gA,n}(\rja)=\gen{1}$.%
\item 
\begin{enumerate}
\item La \klg  $\gA$  est dite \textsl{présentée comme lisse de base} (\textsl{standard smooth} chez Stacks) si $s\leq n$ et le premier mineur $s\times s$ de $\rja$ est \iv dans $\gA$.
\item La \klg $\gA$  est dite \textsl{lisse} s'il existe un \sys $\vr$ d'\eco de $\gk$ telle que chacune des \algs localisées en $v_j$\footnote{Il s'agit de la $\gk[1/v_j]$-\alg $\gA[1/\rho(v_j)]$.} admet une \pn lisse de base.
\item La \klg $\gA$  est dite \textsl{élémentairement lisse} si $\rja$ est \iv à gauche, \cad encore si $\cD_{\gA,s}(\rja)=\gen{1}$. 
\end{enumerate}
\item 
\begin{enumerate}
\item La \klg $\gA$  est \textsl{présentée comme étale de base} si $s=n$ et $\det(\rja)\in\Ati$, i.e. $\cD_{\gA,n}(\rja)=\gen{1}$. On dit aussi dans ce cas que $\gA$ est une \textsl{\agN}. 
\item La \klg~$\gA$  est dite \textsl{étale} s'il existe un \sys $\vr$ d'\eco de $\gk$ telle que chacune des \algs localisées en $v_j$ admet une \pn étale de base.
%
\end{enumerate}
\end{enumerate}
\end{definition}

\begin{lemma}[stabilité par produit fini]\label{lem-Om-prodfin}
Soient $\gA_1$, \dots, $\gA_r$ des \klgs \pf et $\gA=\prod_{i=1}^r\gA_i$.
L'\alg $\gA$ est nette (resp. lisse, étale) \ssi chacune des $\gA_i$ est nette (resp. lisse, étale).
\end{lemma}

\begin{lemma} \label{lemAGN0}
Si $\gk$ est un \cdi, toute \alg \eal est une \alg étale.
\end{lemma}
%
\begin{proof}
Le résultat est clair si $\gA=\kx=\aqo\kX f$ avec $f$ \spl. On conclut ensuite en utilisant  le lemme \ref{lem-sanscarre}, le \thref{thEtalePrimitif} (point \textsl{2}) et le lemme \ref{lem-Om-prodfin}.
\end{proof}
%

\begin{lemma}[extension des scalaires] \label{lem-LEN-eds}~\\
 Soit $\gA$ une \klg \pf, et $\alpha:\gk\to\gk'$ un morphisme d'extension des scalaires. Soit $\gA'=\alpha\ist(\gA)\simeq \gk'\te_\gk\gA$.
\begin{enumerate}
\item Si $\gA$ est nette (resp. lisse, étale) sur $\gk$, il en va de même pour $\gA'$ sur $\gk'$.
\item Supposons que $\gk'$ est \fpte sur $\gk$. Si $\gA'$ est nette (resp. lisse, étale) sur~$\gk'$, il en va de même pour $\gA$ sur $\gk$.
\end{enumerate}
\end{lemma}
\begin{plcc} \label{plcc.LEN}
Soit $\gA$ une \klg \pf, et $\sn\in\gk$ \com. Notons $\gk_i=\gk[1/s_i]$
Alors $\gA$ est nette (resp. lisse, étale) sur $\gk$ \ssi chacune des \algs $\gk_i\te_\gk\gA$ est nette (resp. lisse, étale) sur~$\gk_i$.  
\end{plcc}

\section{Un \tho décisif: \algs nettes sur un \cdi}\label{secthdecisif}

Avant d'aborder le \tho décisif nous avons besoin de résultats préliminaires sur la netteté pour les \algs \stfes sur un \cdi.
Nous commençons par un lemme crucial.

\begin{lemma} \label{lemNetstf}
Soit $\gK\to\gA$ une \alg nette et strictement finie sur un \cdi $\gK$. Alors $\gA$ est réduite.  
\end{lemma}
%
Le résultat beaucoup plus fort suivant (lemme \ref{lemNetste}) est donné dans l'article \cite[Une algèbre libre finie nette est traciquement étale, 2025]{KN-Quitte}. Une démonstration plus simple est sans doute possible pour le lemme \ref{lemNetstf}. C'est la raison pour laquelle nous nous basons dans la suite sur ce dernier lemme.

\begin{lemma} \label{lemNetste}
Soit $\gk\to\gA$ une \alg nette libre de rang fini $>0$ sur un anneau commutatif $\gk$. Alors $\gA$ est \eal.  
\end{lemma}

Le \tho de structure  suivant pour les \Klgs \eals sur un \cdi regroupe les résultats du \thref{corlemEtaleEtage} et du \thref{thEtalePrimitif} pour les points~\textsl{1} à \textsl{4}. 
\begin{theorem} \label{th-struc-ste-cdi} 
Soient $\gK$ un \cdi non trivial et $\gA$ une \Klg \stfe. \Propeq
\begin{enumerate}
\item $\gA$ est \eal.
\item $\gA$ est engendrée par des \elts \spls sur $\gK$.
\item Tous les \elts de $\gA$ sont \spls sur $\gK$.
\item $\gA$ est isomorphe à un produit fini de \Klgs $\aqo{\KX}{g_i}$ pour des \pols  unitaires \spls $g_i$.
\item $\gA$ est nette.   
\end{enumerate}
En outre, si $\gK$ est infini ou si $\gA$ est un corps, $\gA$ est isomorphe à une \alg $\gK[x]=\aqo \KX g$ où 
$g$ est un \pol unitaire \spl de $\KX$ (\thref{thEtalePrimitif} de l'\elt primitif). 
\end{theorem}
\note Dans le point \textsl{4}, en posant $g:=g_1\cdots g_r$, si les $g_i$ sont deux à deux étrangers,  on~a un \iso canonique $\gA\simeq \aqo \KX g\simeq \prod_{i=1}^r\aqo\KX{g_i}$.\eoe
\begin{proof} Il reste à démontrer l'\eqvc avec le point \textsl{5}.

\noindent Si $\gA$ est \eal, elle est étale (lemme \ref{lemAGN0}), à fortiori  nette.

\noindent Supposons $\gA$ nette. Démontrons que $\gA$ est \eal.
L'\alg $\gA$ est \zedre d'après le lemme \ref{lemNetstf}.

\noindent Voici tout d'abord une \demo utilisant le principe du tiers exclu.

\noindent Si $\gA$ est un \cdi, on écrit $\gA\simeq\Aqo\KX f=\Kx$ pour un \pol \ird~$f$.
Si $f'\neq 0$ alors $\gen{f,f'}=\gen{1}$. Dans ce cas l'\alg $\gA$ est nette et \eal.
Si $f'=0$, la \cara est égale à un nombre premier~$p$ et $f(X)=g(X^p)$ pour un \pol $g$.  Donc $\Om{\gK}{\gA}$ est libre de rang $1$ avec pour base $\rd x$. L'\alg $\gA$ n'est pas nette. Le cas corps discret avec $f'=0$ est donc  impossible.

\noindent Si $\gA$ est n'est pas un \cdi, comme elle est réduite, elle n'est pas connexe. Il y a un \idm $e\neq 0,1$  et $\gA\simeq\aqo\gA e\times \aqo\gA {1-e}$. Les deux \algs sont des \algs nettes $\gA_1$ et $\gA_2$ avec $\rg_\gK(\gA_i)< \rg_\gK(\gA)$ ($i=1,2$). On conclut par \recu sur le rang de $\gA$.

\noindent Voici maintenant un décryptage \cof de la \demo classique. 

\noindent Notons $m=[\gA:\gK]$ la dimension de $\gA$ comme \Kev, notée aussi $\rg_\gK(\gA)$. Si $m=0$ ou $1$, $\gA$ est \eal. Supposons $m>1$.
\\
Si $p=_\gK0$ pour un $p\leq m$, notons $\gK_1$ la clôture parfaite de $\gK$ 
et~$\gA_1$ l'\eds de $\gA$ à $\gK_1$.
Modulo l'identification de $\gK$ à un sous-corps de $\gK_1$ on~a $\Disc_{\gA/\gK}=\Disc_{\gA_1/\gK_1}$, donc $\gA$ est \eal sur $\gK$ \ssi $\gA_1$ est \eal sur $\gK_1$. Nous pouvons donc supposer \spdg que $\gK$ est un corps parfait. 
\\
On a $\gA=\Kxn$, avec $n\geq 1$ et les~$x_i\notin \gK$  \agqs sur~$\gK$. Si les~$x_i$ sont \spls sur~$\gK$\footnote{Pour le décider il suffit de calculer le \discri du \polmin de $x_i$}, l'\alg est \eal d'après le point \textsl{2}.\\ 
Sinon, par exemple avec $x_1$ non \spl, on calcule le \polmin~$f(Y)$ de $x_1$. On a $f'\neq 0$ car sinon $f(Y)=g(Y^p)$ pour $p=_\gK0$, et comme~$\gK$ est parfait $g(Y^p)=g_1(Y)^p$ pour un $g_1\in\KY$, donc $g_1(x_1)=0$
avec $\deg(g_1)<\deg(f)$.
\\
Comme $1\notin\gen{f,f'}$ on peut factoriser $f$ sous la forme $f=f_1f_2$ avec $\gen{f_1,f_2}=\gen{1}$, $\gen{f_1,f'}=\gen{1}$, et $\deg(f_i)< \deg(f)$ ($i=1,2$).
On a alors un \idm $e\neq 0,1$ tel que $\gen{e}=\gen{f_1(x_1)}$ et
$\gen{1-e}=\gen{f_2(x_1)}$. L'\alg~$\gA$ est donc le produit de deux \algs $\gA_1$ et $\gA_2$ avec $\rg_\gK(\gA_i)< \rg_\gK(\gA)$ ($i=1,2$). On conclut par \recu sur le rang de $\gA$.
\end{proof}

Voici maintenant le \tho décisif convoité.
\begin{theorem}[\algs nettes sur un \cdi] \label{th-net-cdi-etale}~\\
Sur un \cdi\footnote{Le \tho est \egmt valable dans le cas où $\gK$ est seulement supposé \zedr. On l'obtient en utilisant la machinerie \lgbe \elr des anneaux \zedrs décrite dans \cite[section IV-8]{ACMC}.}, toute \Klg nette est \stfe, \eal.\\
En résumé, une \Klg \pf sur un \cdi est nette \ssi elle est \eal.  
\end{theorem}
Ainsi toute \Klg étale au sens de la \dfn \gnle \ref{defi-LNE} est \eal au sens de la \dfn \ref{defi1Etale}.

Notons que vu le \thref{th-struc-ste-cdi} il suffit de démontrer qu'une \Klg nette est \stfe. D'après   \cite[\tho IV-8.16]{ACMC}, cela revient à dire que si une \Klg \pf $\gA$ est nette, sa dimension de Noether est $\leq 0$.

\smallskip Il nous reste donc à démontrer le \tho suivant.

\begin{theorem} \label{thAlgNetteSruCdi}
Soient $\gK$ un \cdi,
 $f_1$, \dots, $f_s\in\KXn$, \hbox{et $\gA=\Kux=\aqo{\KuX}{\uf}$} la \Klg quotient. Si $\Omega_{\gA/\gK}=0$ alors la dimension de Noether du \syp est $\leq 0$. 
\end{theorem}
%
\begin{proof} Puisque la dimension de Noether $r$ se calcule, on peut raisonner par l'absurde et supposer que $r>0$. 
On pose 
\begin{itemize}
\item $\gB=\gK[\Yr]$, $S=\gB\setminus \so0$;
\item  $\gL=S^{-1}\gB=\gK(\Yr)$, $Z_i=Y_{r+i}$ ($i=1,\dots,t=n-r$);
%
%
\item   $\gC=S^{-1}\gA$.
\end{itemize}
Notons que $\gA=\gB[Z_1,\dots,Z_t]/\fa$ où $\fa$ est un \itf tel que $\fa\cap S=\emptyset$. On~a alors (voir le lemme \ref{propOmkBS}) 
\[
0=S^{-1}\Omega_{\gA/\gK}=\Omega_{\gC/\gK}.
\]  
Comme
$\gA$ est un \Bmo \pf et comme  $\fa\cap S=\emptyset$, l'\alg $\gC$ est un \Lev non nul libre de dimension finie. Les morphismes $\gK\to\gL\to\gC$ 
donnent la suite exacte classique de $\gC$-modules (voir le lemme~\ref{propOmABC})
\[
\gC\otimes_\gK\Omega_{\gL/\gK} \lora \Omega_{\gC/\gK} \lora \Omega_{\gC/\gL}\lora 0 
\]
Puisque $\Omega_{\gC/\gK}=0$ on obtient $\Omega_{\gC/\gL}=0$. Comme $\gL$ est un \cdi et que $\gC$ est un \Lev de dimension finie, $\gC$ est \eal sur~$\gL$. Comme $\gL$ est infini, $\gC=\gL[x]=\aqo {\gL[X]} f$ avec $f$ un \pol unitaire \spl.\\
Considérons alors la dérivée partielle par rapport à $X_1$: $\partial_1:\gL\to\gL$. C'est une $\gK$-dérivation non nulle. Montrons que nous pouvons la prolonger en une $\gK$-dérivation $\partial_1:\gC\to\gC$.
Il suffit pour cela de définir $\partial_1 x$, lequel est soumis à l'unique condition 
\[
(\partial_1 x)\,f'(x)+\partial_1 f=0.
\]
Par exemple si $f(X)=X^3+\alpha X^2+\beta X+ \gamma$, on doit réaliser la condition 
$$(\partial_1 x)(3x^2+2\alpha x+\beta)+ x^2\partial_1\alpha+ x\partial_1\beta+\partial_1\gamma=0
.$$
Or  $f'(x)$ est inversible dans $\gC=\gL[x]$; il suffit donc de poser  
$\partial_1 x:=-(\partial_1 f)/f'(x)$.
\\
Ainsi on~a construit une $\gK$-dérivation non nulle $\partial_1:\gC\to\gC$. Cela implique que \hbox{$\Omega_{\gC/\gK}\neq 0$}.
En bref, les hypothèses $\Omega_{\gA/\gK}= 0$ et $r>0$ sont incompatibles.
\end{proof}
%



\end{document}